%% file: main.tex
\documentclass[11pt,a4paper]{article}
\usepackage{graphicx} 

\input{preamble.tex}

\graphicspath{{figures/}}

\title{Graphical C(3)-T(6) implies CAT(0)
}
\author{Huaitao Gui}
\date{}

\begin{document}

\maketitle

\begin{abstract}
Graphical small cancellation extends the classical small cancellation theory and provides a powerful method for constructing groups with interesting features. In the classical setting, $C(3)$-$T(6)$ small cancellation complexes are known to admit locally CAT(0) metrics. In this paper, we construct locally CAT(0) metrics for graphical $C(3)$-$T(6)$ complexes.
\end{abstract}

\section{Introduction}

Graphical small cancellation is a generalization of the classical small cancellation theory. It provides a method for constructing groups whose Cayley graphs contain certain prescribed graphs, giving rise to groups with interesting properties; see, for example, \cite{Gro-random, OW-Kazhdan, Osajda20, Coulon-Gruber, Osajda-residual}.

A graphical presentation of a group has the form $G=\langle f: \Gamma \rightarrow \Theta \rangle$, where $\Theta$ is a connected graph, $f:\Gamma \rightarrow \Theta$ is a graph immersion (that is, a locally injective graph homomorphism), and $G$ is the fundamental group of $\Theta$ modulo the normal subgroup generated by the images of all immersed cycles in $\Gamma$. 
The use of such presentations goes back to the work of Rips and Segev, who constructed torsion-free groups without the unique product property (see \cite{RS-torsion,Steenbock-torsion}). Small cancellation theory over graphical presentations was later used explicitly by Gromov \cite{Gro-random} in his construction of Gromov's monster---a group that contains an expander graph in a weak sense.

There are two 2-dimensional combinatorial complexes naturally associated to such a presentation: the thickened graphical complex $X_t$, where 2-cells are attached to $\Theta$ to fill the cycles coming from $\Gamma$, and the non-thickened graphical complex $X$, where cones over the connected components of $\Gamma$ are attached to $\Theta$ instead. 
In the latter case, each cone is triangulated so that $X$ admits a 2-dimensional combinatorial complex structure where all 2-cells are triangles.
Both complexes have $G$ as their fundamental group. Precise definitions will be given in Section~\ref{sec: prelimiaries}. 

Since Gromov's introduction of hyperbolic groups \cite{Gro87}, the interaction between small cancellation and nonpositive or negative curvature has been an active area of study.
In this article, we focus on non-metric small cancellation conditions.
A classical result by Gersten and Short \cite{GS90} states that classical $C(p)$-$T(q)$ small cancellation groups are hyperbolic when $1/p+1/q<1/2$. 
In the borderline case $1/p+1/q=1/2$, that is, when $(p,q)\in \Set{(3,6),(4,4),(6,3)}$, various links between small cancellation and metric or combinatorial nonpositive curvature have been established: 
classical $C(6)$ groups are systolic \cite{Wise03}, a result extended to the graphical setting in \cite{OP18}; 
graphical $C(4)$-$T(4)$ groups are Helly \cite{Hoda-quadric,CCGHO25}; 
and classical $C(3)$-$T(6)$ groups are CAT(0) \cite{Pride88, BH13, Duda21}. 
In this paper, we extend the last result to the graphical setting.

\begin{theorem}\label{thm: simplicial + T(6) implies CAT(0)}
    Let $G=\langle f: \Gamma \rightarrow \Theta\rangle$ be a $T(q)$ graphical small cancellation presentation with $Girth(\Gamma)\geq p$, where $Girth(\Gamma)$ denotes the length of a shortest cycle in $\Gamma$, and let $X$ be the non-thickened graphical complex associated to this presentation. 
    \begin{enumerate}[label=(\arabic*)]
        \item If $(p,q)=(3,6)$, then $X$ admits a locally CAT(0) metric.
        \item If $(p,q)\in \{(4,5), (3,7)\}$, then $X$ admits a locally CAT(-1) metric.
    \end{enumerate}
\end{theorem}

The proofs rely on the link condition for 2-dimensional complexes. We first show that small cancellation conditions provide lower bounds on the girth of the links in a graphical complex $X$. We then endow every 2-cell in $X$ with a Euclidean or hyperbolic metric and show that $X$ satisfies the link condition.

Note that any graphical presentation satisfying the assumptions of either statement of Theorem~\ref{thm: simplicial + T(6) implies CAT(0)} also satisfies the corresponding $C(p)$ small cancellation condition. 
If instead we start with the weaker assumption that a graphical presentation is $C(p)$-$T(q)$, then we can sightly modify the proof of Theorem~\ref{thm: simplicial + T(6) implies CAT(0)} so that the conclusions of the theorem still hold; see Corollary~\ref{cor: C(3) + T(6) implies CAT(0)}.

We now consider group actions on simply connected graphical complexes. Let $\widetilde{X}$ be a simply connected non-thickened graphical complex associated to a $C(p)$-$T(q)$ graphical presentation with $(p,q)\in\{(3,6),(4,5)\}$. Then it can be subdivided into a triangle complex $\widetilde{X}'$ (i.e., a 2-dimensional simplicial complex) and endowed with a piecewise Euclidean metric that is CAT(0), recurrent (in the sense of \cite[Definition 2.1]{OP21-recurrent}), and has rational angles (in the sense of \cite{NOP22-torsion}). 

Let $G$ be a group acting on $\widetilde{X}$ by automorphisms (see the definition in Section \ref{sec: prelimiaries}), then it also acts on $\widetilde{X}'$ by combinatorial automorphisms and by isometries with respect to the given metric. Therefore, either by the recurrent structure and \cite[Main Theorem]{OP21-recurrent} or by the CAT(0) structure and \cite[Theorem A]{OP22-Tits}, we obtain:  

\begin{corollary}\label{cor:Tits alternative}
    If the action is almost free (that is, there is a uniform bound on the cardinalities of all cell stabilizers), then $G$ satisfies the Tits alternative; more precisely, every finitely generated subgroup of $G$ is either virtually cyclic, virtually $\mathbb{Z}^2$, or contains a nonabelian free group. In particular, graphical $C(p)$-$T(q)$ small cancellation groups satisfy the Tits alternative for $(p,q)\in \{(3,6),(4,5)\}$.
\end{corollary}  

Additionally, since $\widetilde{X}'$ is a CAT(0) triangle complex with rational angles, applying \cite[Theorem 1.1]{NOP22-torsion}, we obtain the following, which proves a particular case of \cite[Meta-Conjecture, Conjecture]{HO-elliptic}:

\begin{corollary}\label{cor:fixed point}
    If $G$ is finitely generated and every element of $G$ fixes a cell in $\widetilde{X}$, then the entire group $G$ fixes a cell in $\widetilde{X}$. 
\end{corollary}

The paper is organized as follows. In Section \ref{sec: prelimiaries}, we recall the relevant definitions from graphical small cancellation theory; in Section \ref{sec: Small cancellation condition T}, we discuss several implications of the graphical small cancellation condition $T$, in particular how it bounds the girth of some links in a graphical complex; finally, in Section \ref{sec: Proof of the results}, we prove Theorem \ref{thm: simplicial + T(6) implies CAT(0)} and the corollaries. 

\section{Preliminaries}
\label{sec: prelimiaries}

In this section, we give basic definitions in graphical small cancellation theory. Our exposition follows closely \cite[Section 6]{OP18} and \cite[Section 6.4]{CCGHO25}, where graphical C(6) and $C(4)$-$T(4)$ groups are studied, respectively.

A continuous map between CW complexes $X$ and $Y$ is called \textbf{combinatorial} if it maps each open cell of $X$ homeomorphically onto an open cell of $Y$. A CW complex is \textbf{combinatorial} provided that the attaching map of each open cell of $X$ is combinatorial for a suitable subdivision. 

Throughout this article, all complexes and maps are assumed to be combinatorial. We refer to an $n$-dimensional combinatorial complex simply as an $n$-complex, and we identify graphs with 1-complexes. We also refer to 0-cells of a complex as vertices and 1-cells as edges.

For $n\in \mathbb{N}$, let $P_n$ denote a \textbf{path graph} with $n$ edges, whose vertices are labeled from $0$ to $n$ in order. A \textbf{path} in a complex $X$ is a map $\gamma: P_n \rightarrow X$ for some $n$, and $n$ is called the \textbf{combinatorial length} of $\gamma$, denoted by $|\gamma|$. 
The path $\gamma$ is \textbf{closed} if $\gamma(0)=\gamma(n)$. 
A path $\gamma': P_m\rightarrow X$ is called a \textbf{subpath} of $\gamma$ if there exists an injective, order-preserving map $\iota: P_m\rightarrow P_n$ such that $\gamma' = \gamma \circ \iota$. The notations for a path of length 1 and a directed edge will be used interchangeably, with no distinction made between the two. 

For $n\in \mathbb{N}_{>0}$, let $C_n$ denote a \textbf{cycle graph} with $n$ edges, whose vertices are labeled by $\mathbb{Z}/n\mathbb{Z}$. A \textbf{cycle} in $X$ is a map $\tau: C_n \rightarrow X$ for some $n$, and its \textbf{combinatorial length} is denoted $|\tau|=n$. 
Let $R_n$ denote an $n$-gon, i.e., a disk whose 1-skeleton is $C_n$. Every 2-cell in $X$ can be represented by a map $R_n\rightarrow X$ for some $n$. 
We write $P$, $C$, or $R$ for the domain of a path, a cycle, or a representative of a 2-cell, respectively, when the size is not specified. 

For $m\in \mathbb{N}$ and $n\in \mathbb{Z}/n\mathbb{Z}$, fix a map $\gamma_{m,n}: P_m\rightarrow C_n$, which is locally injective and extends the canonical map from $\{0,...,m\}$ to $\mathbb{Z}/n\mathbb{Z}$. We say that a graph $\Gamma$ is \textbf{path-cycle extensible} if every immersed path (i.e., locally injective path)  in $\Gamma$ factors through an immersed cycle. That is, for every immersed path $\gamma: P_m\rightarrow \Gamma$, there exists an immersed cycle $\tau: C_n\rightarrow \Gamma$ such that $\gamma = \tau\circ \gamma_{m,n}$. 

\begin{remark} \label{rm: path-cycle extensible}
    Note that a graph $\Gamma$ is path-cycle extensible if and only if every vertex of $\Gamma$ lies in some immersed cycle. Therefore, for any connected graph $\Omega$, we can consider the subgraph $\Omega'$ induced by all the vertices of $\Omega$ that lie in immersed cycles. If $\Omega$ is a tree, then $\Omega'$ is empty; if $\Omega$ is not a tree, then $\Omega'$ is a connected, path-cycle extensible subgraph that carries the same fundamental group as $\Omega$. For a finite graph, being path-cycle extensible is equivalent to having no vertices of degree 1. 
\end{remark}

Given a graph immersion $f:\Gamma \rightarrow \Theta$ with $\Theta$ connected, we define a group $G$ as follows.
Decompose the graph $\Gamma = \coprod_i \Gamma_i$, where each graph $\Gamma_i$ is a connected component. 
Fix a vertex $u\in \Theta$, and for each $i$, choose a vertex $u_i\in \Gamma_i$ and a path $\gamma_i$ in $\Theta$ from $u$ to $f(u_i)$. 
Define a homomorphism $f_i: \pi_1(\Gamma_i,u_i)\rightarrow \pi _1(\Theta,u)$, $[\gamma]\mapsto [\gamma_i\cdot f\circ\gamma \cdot \overline{\gamma_i}]$, where $\gamma$ is a closed path in $\Gamma_i$ based at $u_i$, $\cdot$ denotes path concatenation, and $\bar{\cdot}$ denotes path reversal. 
Then define $G \coloneqq \pi_1(\Theta,u) \big/ \big\langle\big\langle \bigcup_i f_i(\pi_1(\Gamma_i,u_i)\big\rangle\big\rangle$. 
This group is well-defined up to isomorphism, independent of the choices of the basepoints and the connecting paths. 
$f:\Gamma\rightarrow \Theta$ is called a \textbf{graphical presentation} of $G$, denoted $G=\langle f:\Gamma \rightarrow \Theta \rangle$.

There are two natural ways to construct a 2-complex with fundamental group $G=\langle f:\Gamma \rightarrow \Theta \rangle$. 
First, let the 1-skeleton be $\Theta$. 
For every immersed cycle $C\rightarrow \Gamma$, attach a 2-cell to $\Theta$ along the composition $C\rightarrow \Gamma \rightarrow \Theta$. 
The resulting complex is called the \textbf{thickened graphical complex} associated to $G$, and we denote it by $X_t$. 
Alternatively, for each connected component $\Gamma_i$ of $\Gamma$, attach the simplicial cone $Cone(\Gamma_i)$ to $\Theta$ via the restriction of $f$ on $\Gamma_i$. The resulting complex is called the \textbf{non-thickened graphical complex} associated to $G$, and we denote it by $X$. Note that every 2-cell in $X$ is a triangle. An \textbf{automorphism} of a non-thickened graphical complex $X$ is a combinatorial automorphism that preserves $\Theta$. 

Throughout the paper, we will assume that $\Gamma$ is path-cycle extensible.
This involves no loss of generality: for each connected component $\Gamma_i$ of $\Gamma$, we may replace $\Gamma_i$ by the path-cycle extensible subgraph $\Gamma'_i$ described in Remark~\ref{rm: path-cycle extensible}. Let $\Gamma' = \coprod_i \Gamma'_i$. Then the graphical presentations $\langle f: \Gamma \rightarrow \Theta\rangle $ and $\langle f|_{\Gamma'}: \Gamma' \rightarrow \Theta \rangle$ define the same group, the same thickened graphical complex, and homotopy equivalent non-thickened graphical complexes.  

\begin{definition}\label{def:piece}
    A \textbf{piece} in a thickened graphical complex $X_t$ is a nontrivial immersed path $\gamma: P\rightarrow \Theta$ in the graph $\Theta$ that admits two distinct lifts to $\Gamma$; that is, there exist two distinct paths $\gamma_1, \gamma_2:P\rightarrow \Gamma$ such that $f\circ \gamma_1 = f\circ \gamma_2 = \gamma$ but $\gamma_1\neq \gamma_2$. 
\end{definition}

Definition~\ref{def:piece} agrees with \cite[Definition 6.3]{OP18}, provided that the graph $\Gamma$ is path-cycle extensible. Note that if $\gamma$ is a piece in the complex $X_t$, and $\gamma'$ is a nontrivial subpath of $\gamma$, then $\gamma'$ is also a piece. This follows from the fact that the map $f:\Gamma \rightarrow \Theta$ is an immersion.

A \textbf{disk diagram} is a finite, contractible 2-complex that embeds in the plane. A \textbf{disk diagram in a thickened graphical complex} $X_t$ is a map from such a disk diagram $D$ to $X_t$.  
A \textbf{piece in a disk diagram} $D$ is an immersed nontrivial path $P\rightarrow D$ that factors in two distinct ways through 2-cells of $D$; that is, there exist 2-cells $R_m\rightarrow D$ and $R_n\rightarrow D$ such that the path $P\rightarrow D$ factors both as $P\rightarrow R_m \rightarrow D$ and $P\rightarrow R_n \rightarrow D$, and there is no isomorphism $R_m\rightarrow R_n$ making the following diagram commute:
\[
\begin{tikzcd}
    P \arrow[r] \arrow[d] & R_m \arrow[d] \arrow[ld]\\
    R_n \arrow[r] & D
\end{tikzcd}
\]

\begin{example}
    A \textbf{jasmine} with $n$ petals is a disk diagram that is isomorphic to $\coprod_{i=1}^n R_{m_i}/\sim$, where $m_i\geq 3$ for all $i$, and the identification is given by \[ \gamma_{m_{i-1},m_{i-1}}\big((m_{i-1},m_{i-1}-1)\big) \sim \gamma_{m_i,m_i}\big((0,1)\big) \]
    for $i$, with the convention $m_0\coloneqq m_n$. 
    When $n\geq 3$, the pieces in such a disk diagram correspond precisely to the interior edges. See Figure \ref{fig:Examples of jasmine diagrams} for examples. 
    \begin{figure}[h]
        \centering
        \includegraphics[width=0.7\linewidth]{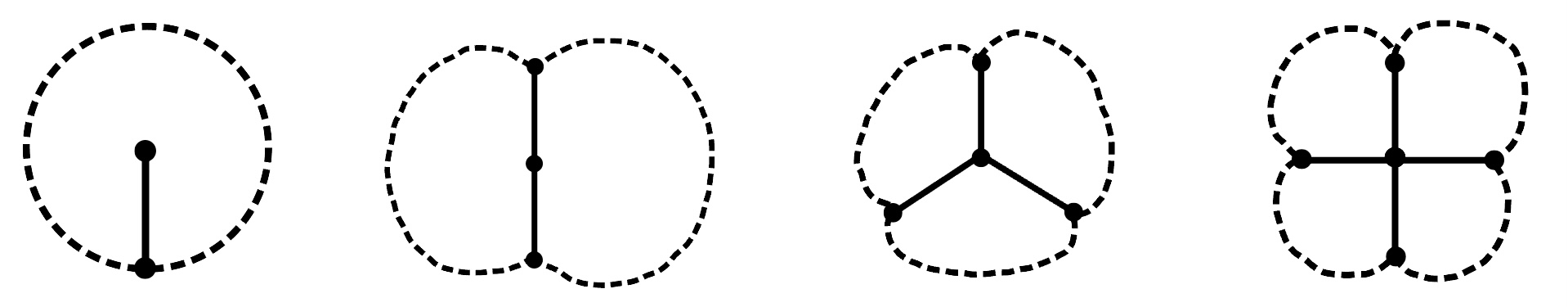}
        \caption{Jasmine diagrams with one to four petals}
        \label{fig:Examples of jasmine diagrams}
    \end{figure}
\end{example}

A disk diagram $D\rightarrow X_t$ in a thickened graphical complex is \textbf{reduced} if for every piece $P\rightarrow D$ in the disk diagram, the composition $P\rightarrow D\rightarrow X_t$ is again a piece. 

\begin{definition}
    Let $p,q\in \mathbb{N}_{>0}$. The thickened graphical complex $X_t$ associated to a graphical presentation $G=\langle f:\Gamma \rightarrow \Theta\rangle$ satisfies:
    \begin{itemize}
        \item the \textbf{$C(p)$ small cancellation condition} if there is no immersed cycle $C\rightarrow\Gamma$ such that the composition $C\rightarrow \Gamma \rightarrow \Theta$ is a concatenation of fewer than $p$ pieces in $X_t$;
        \item the \textbf{$T(q)$ small cancellation condition} if for every reduced disk diagram $D\rightarrow X_t$, there is no interior vertex of $D$ whose degree is greater than 2 and less than $q$. 
    \end{itemize}  
    A presentation $\langle f:\Gamma \rightarrow \Theta \rangle$ and its associated non-thickened graphical complex $X$ are said to satisfy the $C(p)$ (respectively, $T(q)$) condition if the associated thickened complex $X_t$ does.
\end{definition}

\begin{remark}
     If $C_n\rightarrow \Gamma$ is an immersed cycle such that the composition $P_n\xrightarrow{\gamma_{n,n}} C_n\rightarrow \Gamma \rightarrow \Theta$ is a concatenation of multiple identical paths, then the path $P_n \rightarrow C_n \rightarrow \Gamma \rightarrow \Theta$ is a piece. However, this situation is excluded already by the $C(2)$ condition. For a less restrictive condition that allows proper powers, see for example \cite[Definition 1.3]{Gruber15}.
\end{remark}

\section{Small cancellation condition $T$}
\label{sec: Small cancellation condition T}

In this section, we take a closer look at the $T(q)$ small cancellation condition and discuss some of its implications.
We begin by establishing a graphical analogue of an observation of Pride \cite[p.~165]{Pride88}, namely that every piece in a classical $T(q)$ small cancellation complex has length 1 when $q\geq 5$. 
The proof provided here follows essentially the lines of \cite[Lemma~3.5]{McCammond-Wise}.

\begin{lemma} \label{Lem: pieces have length 1}
    Let $\langle f:\Gamma \rightarrow \Theta\rangle$ be a $T(q)$ graphical small cancellation presentation with $q\geq 5$, then every piece in the associated thickened graphical complex $X_t$ has length 1.
\end{lemma}

\begin{proof}
    Suppose there exists a length 2 piece $\gamma: P_2\rightarrow \Theta$ in $X_t$. By definition, there are two distinct immersed paths $\gamma_1, \gamma_2: P_2 \rightarrow \Gamma$ such that $f\circ \gamma_i=\gamma$ for $i=1,2$. Since $\Gamma$ is path-cycle extensible, each $\gamma_i$ factors through an immersed cycle, i.e., there exist immersed cycles $\tau_1: C_m\rightarrow \Gamma$ and $\tau_2: C_n\rightarrow \Gamma$ such that $\gamma_1 = \tau_1\circ \gamma_{2,m}$ and $\gamma_2 = \tau_2\circ \gamma_{2,n}$. We may assume $m,n\geq 3$, because if, for instance, $m<3$, we can replace $\tau_1$ by the composition $C_{3m}\rightarrow C_m\xrightarrow{\tau_1} \Gamma$, where $C_{3m}\rightarrow C_m$ is a covering map. Let $R_m\rightarrow X_t$ and $R_n\rightarrow X_t$ be 2-cells in $X_t$ with boundary maps $f\circ \tau_1$ and $f\circ \tau_2$, respectively. Let $D$ be a jasmine diagram with four petals, constructed by taking the disjoint union of two copies of $R_m$, denoted $R'_m$ and $R''_m$, and two copies of $R_n$, denoted $R'_n$ and $R''_n$, and identifying the following pairs of edges: $\gamma'_{2,m}\big((1,2)\big)\sim \gamma'_{2,n}\big((1,2)\big)$, $\gamma'_{2,n}\big((0,1)\big)\sim \gamma''_{2,m}\big((0,1)\big)$, $\gamma''_{2,m}\big((1,2)\big)\sim \gamma''_{2,n}\big((1,2)\big)$, $\gamma''_{2,n}\big((0,1)\big)\sim \gamma'_{2,m}\big((0,1)\big)$, see Figure \ref{fig:A disk diagram with four petals}.
    \begin{figure}
        \centering
        \includegraphics[width=0.25\linewidth]{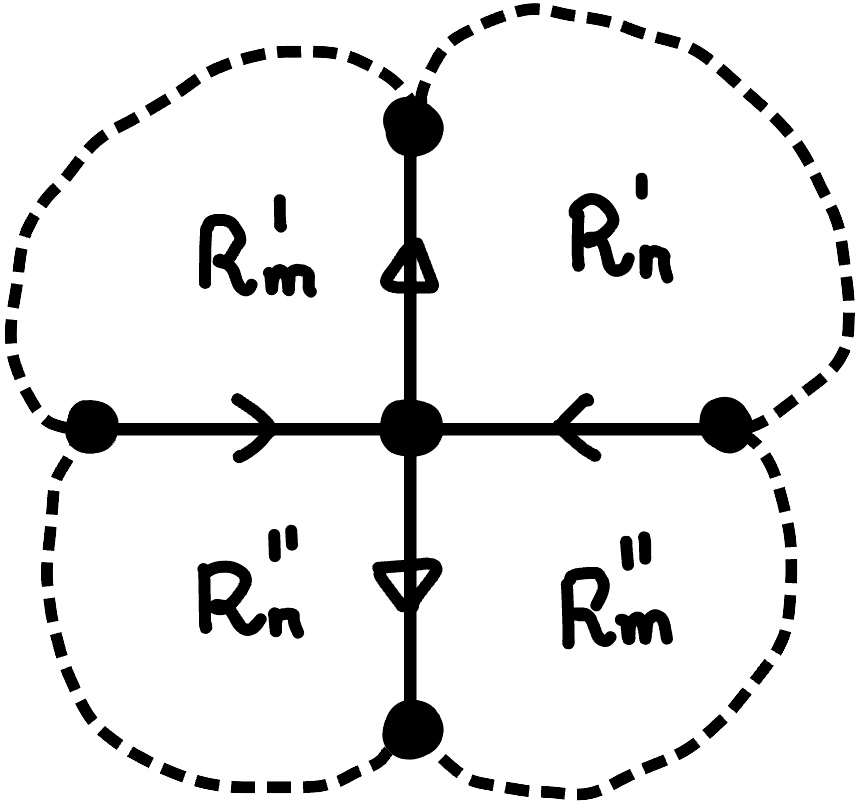}
        \caption{A disk diagram with four petals}
        \label{fig:A disk diagram with four petals}
    \end{figure}Then the maps $R_m\rightarrow X_t$ and $R_n\rightarrow X_t$ induce a map $D\rightarrow X_t$. 
    Since every piece in $D$ maps to a subpath of $\gamma$ or $\bar{\gamma}$ (the reversal of $\gamma$), the image is again a piece in $X_t$. Therefore, the map $D\rightarrow X_t$ defines a reduced disk diagram. 
    However, $D$ has an internal vertex of degree 4, which contradicts the $T(q)$ condition. Hence, it is impossible for a piece of length 2 or greater to exist.
\end{proof}

Next, we discuss the connection between the $T(q)$ condition and the lengths of cycles in certain links of a graphical complex.

Let $Y$ be a 2-complex and let $v\in Y$ be a vertex. The \textbf{link} of $v$ in $Y$, denoted $Lk(v,Y)$, is the graph defined as follows. First, identify each edge of $Y$ with the unit interval and equip the 1-skeleton $Y^1$ with the induced metric. For $0<\epsilon<\frac{1}{2}$, let $V_\epsilon:=\{y\in Y^1\,|\,d(v,y)=\epsilon\}$. For two distinct points $y,y'\in V_\epsilon$, join $y$ and $y'$ by an edge if and only they are contained in the closure of a common 2-cell of $Y$. Let $E_\epsilon$ denote the collection of all such edges. The resulting graph $(V_\epsilon,E_\epsilon)$ is independent of the choice of $\epsilon$ up to graph isomorphism. Let $Lk(v,Y):=(V_\epsilon,E_\epsilon)$. Intuitively, $Lk(v,Y)$ may be viewed as the intersection of $Y$ with a sufficiently small sphere centered at $v$. Equivalently, if we associate to each edge $\{u,v\}$ of $Y$ two oppositely oriented edges $(u,v)$ and $(v,u)$, then the vertices of $Lk(v,Y)$ correspond bijectively to the oriented edges with initial vertex $v$. Two such vertices are joined by an edge in $Lk(v,Y)$ if and only if the corresponding oriented edges span a corner of a 2-cell of $Y$. 

In the classical setting, it is known (see, for example, \cite[p.~296]{Hill-Pride-Stephen} or \cite[p.611]{McCammond-Wise}) that, under suitable assumptions, a 2-complex satisfies the $T(q)$ small cancellation condition if and only if the link of every vertex contains no immersed cycle whose length is greater than 2 and less than $q$.

In the graphical setting, let $\langle f:\Gamma \rightarrow \Theta\rangle$ be a graphical presentation, and let $X$ be the associated non-thickened graphical complex. We begin by describing the combinatorial structure of $X$.

Write the graph $\Gamma = \coprod_i \Gamma_i$ as the disjoint union of its connected components, called \textbf{relators}. 
Recall that $X$ is constructed by attaching the simplicial cone $Cone(\Gamma_i)$ to the graph $\Theta$ for each relator $\Gamma_i$.
Therefore, the vertex set of $X$ is the disjoint union of the vertex set of $\Theta$ and the collection of cone tips, i.e., $V(X)=V(\Theta)\coprod V_{tip}(X)$. 
Let $E_{tip}(X)$ be the set of all edges of $X$ that are incident to a cone tip. 
Then the edge set of $X$ decomposes as $E(X)=E(\Theta)\coprod E_{tip}(X)$. 
For the tip $v_i$ of a cone $Cone(\Gamma_i)$, the link $Lk(v_i,X)$ is isomorphic to its base $\Gamma_i$. 
For a vertex $v\in V(\Theta)$, the link $Lk(v,X)$ is a bipartite graph, since every corner at $v$ in $X$ is formed by one edge in $E(\Theta)$ and one edge in $E_{tip}(X)$. 
Hence, every cycle in $Lk(v,X)$ has even length.

\begin{lemma} \label{lem:T condition and the links}
    Let $\langle f:\Gamma \rightarrow \Theta\rangle$ be a T(q) graphical presentation with $q\geq 4$. Then for any vertex $v\in V(\Theta)$, every immersed cycle in $Lk(v,X)$ has length either $4$ or at least $2q$. Moreover, when $q\geq 5$, there are no immersed cycles of length 4 in $Lk(v,X)$. Therefore, we have $Girth(Lk(v,X))\geq 2q$.
\end{lemma}

\begin{proof} 

    First, we show that there is no immersed cycle of length 2 in $Lk(v,X)$. 
    Suppose, for contradiction, that such a cycle exists.
    Let one vertex in this cycle represent the directed edge $\vec{e}_1$ from $v$ to the cone tip $v_i$ of $Cone(\Gamma_i)$. 
    Let the other vertex represent the directed edge $\vec{e}_2$ from $v$ to a vertex $u\in V(\Theta)$, as illustrated in Figure \ref{fig:Two cycle in the link}.
    \begin{figure}
        \centering
        \includegraphics[width=0.475\textwidth]{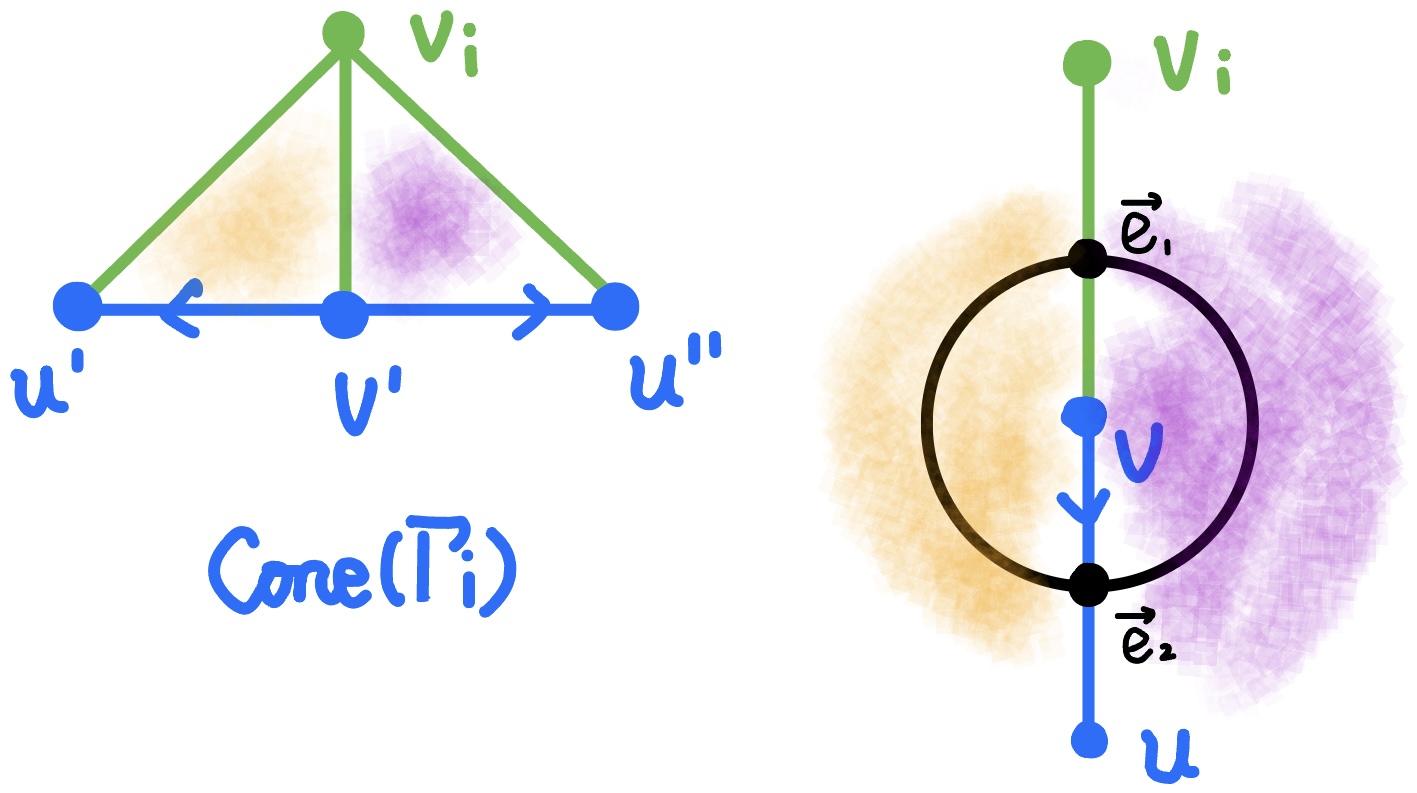}
        \caption{2-cycle in $Lk(v,X)$}
        \label{fig:Two cycle in the link}
    \end{figure}
    Since $\vec{e}_1$ has a unique lift $(v',v_i)$ in $Cone(\Gamma_i)$ for some $v'\in \Gamma_i$, the lifts of the corners represented by the edges of the 2-cycle must also lie in $Cone(\Gamma_i)$ and must be incident to $(v',v_i)$. 
    Denote these lifted corners by $\angle v_i v' u'$ and $\angle v_i v' u''$, for some $u',u''\in \Gamma_i$. 
    Because we began with an immersed cycle in $Lk(v,X)$, the directed edges $(v',u')$ and $(v',u'')$ are distinct.
    However, both map to $\vec{e}_2$, contradicting the assumption that $\Gamma \rightarrow \Theta$ is an immersion. 
    Therefore, there is no immersed cycle of length 2 in $Lk(v,X)$. 

    Second, we show that there is no immersed cycle in $Lk(v,X)$ whose length is greater than 4 and less than $2q$. 
    We discuss only the case of length 6, since the general case is analogous, with heavier notation. 
    Suppose, for contradiction, that $Lk(v,X)$ contains an immersed cycle of length 6, with vertices representing directed edges $\vec{e}_1$, $\vec{e}_2$, $\vec{e}_3$, $\vec{e}_4$, $\vec{e}_5$, $\vec{e}_6$, which connect $v$ to $v_i$, $w$, $v_j$, $x$, $v_k$, $u$, where $v_i, v_j, v_k\in V_{tip}(X)$ and $u,w,x\in V(\Theta)$. In $Cone(\Gamma_i)$, let $(v',v_i)$ be the lift of $\vec{e}_1$, and denote the lifts of the corners corresponding to $\{\vec{e}_1,\vec{e}_6\}$ and $\{\vec{e}_1,\vec{e}_2\}$ by $\angle v_i v' u'$ and $\angle v_i v' w'$ for some $u',w'\in \Gamma_i$. 
    Similarly, in $Cone(\Gamma_j)$, let $(v'',v_j)$ be the lift of $\vec{e}_3$, and denote the lifts of the corners corresponding to $\{\vec{e}_3,\vec{e}_2\}$ and $\{\vec{e}_3,\vec{e}_4\}$ by $\angle v_j v'' w''$ and $\angle v_i v'' x''$, for some $w'', x''\in \Gamma_j$. 
    Finally, in $Cone(\Gamma_k)$, let $(v''',v_k)$ be the lift of $\vec{e}_5$, and denote the lifts of the corners corresponding to $\{\vec{e}_5,\vec{e}_4\}$ and $\{\vec{e}_5,\vec{e}_6\}$ by $\angle v_k v''' x'''$ and $\angle v_k v''' u'''$, for some $x''', u'''\in \Gamma_k$; see Figure \ref{fig:Six cycle in the link}.
    \begin{figure}
        \centering
        \includegraphics[width=0.6\linewidth]{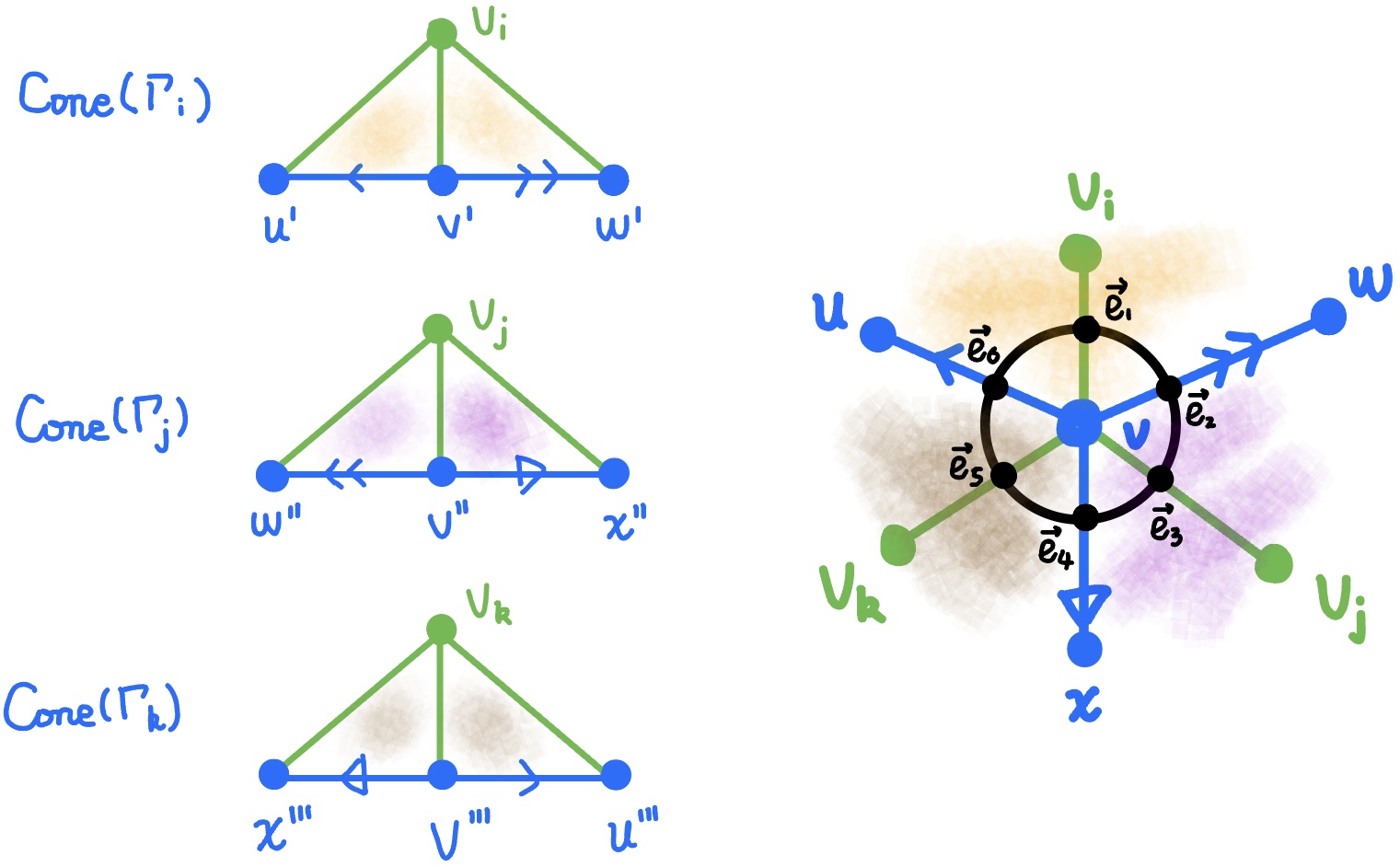}
        \caption{6-cycle in $Lk(v,X)$}
        \label{fig:Six cycle in the link}
    \end{figure}
    Let $\gamma_1: P_2\rightarrow \Gamma_i$ be the path from $u'$ to $w'$ via $v'$; $\gamma_2: P_2\rightarrow \Gamma_j$ the path from $w''$ to $x''$ via $v''$; and $\gamma_3: P_2\rightarrow \Gamma_k$ the path from $x'''$ to $u'''$ via $v'''$. 
    Since we began with an immersed cycle in $Lk(v,X)$, each $\gamma_\ell$ is immersed.
    Because $\Gamma$ is path-cycle extensible, for each $\gamma_\ell$, there exists an immersed cycle $\tau_\ell: C_{m_\ell} \rightarrow \Gamma$, $m_\ell\geq 3$, such that $\tau_\ell\circ \gamma_{2,m_\ell}=\gamma_\ell$. 
    Let $R_{m_\ell}\rightarrow X_t$ be the 2-cell in the thickened graphical complex $X_t$ corresponding to the cycle $C_{m_\ell}\rightarrow \Gamma$. 
    Construct a jasmine diagram $D$ with three petals by taking the disjoint union $\coprod_{\ell=1}^3 R_{m_\ell}$ and identifying $\gamma_{2,m_1}\big((0,1)\big)\sim \gamma_{2,m_3}\big((2,1)\big)$, $\gamma_{2,m_2}\big((0,1)\big)\sim \gamma_{2,m_1}\big((2,1)\big)$, $\gamma_{2,m_3}\big((0,1)\big)\sim \gamma_{2,m_2}\big((2,1)\big)$, as illustrated in Figure \ref{fig:Disk diagram with three petals}.
    \begin{figure}
        \centering
        \includegraphics[width=0.25\linewidth]{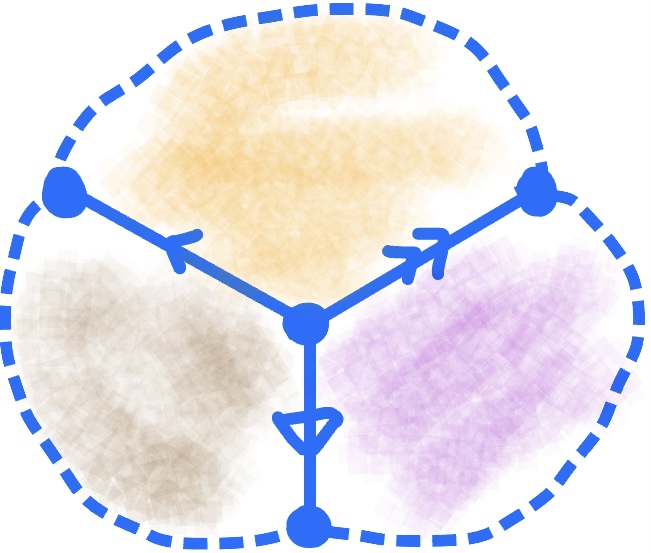}
        \caption{Disk diagram with three petals in $X_t$}
        \label{fig:Disk diagram with three petals}
    \end{figure}
    The inclusions $R_{m_\ell}\rightarrow X_t$ induce a map $D\rightarrow X_t$.
    Because the cycle in $Lk(v,X)$ is immersed, the edges $(v',w')$ and $(v'',w'')$ are distinct, so $\vec{e}_2$ is a piece. 
    Similarly, $\vec{e}_4$ and $\vec{e}_6$ are also pieces. 
    It follows that $D\rightarrow X_t$ is a reduced disk diagram. 
    However, $D$ contains an interior vertex of degree 3, contradicting the assumption that $\langle f:\Gamma \rightarrow \Theta\rangle$ is a $T(q)$ small cancellation presentation with $q\geq 4$. 
    Thus, no immersed cycle of length 6 can occur in $Lk(v,X)$. 
    The general case follows by an analogous argument.

    Finally, we show that no immersed 4-cycle can occur in $Lk(v,X)$ when $q\geq 5$. If such a 4-cycle existed, traversing it twice would produce an immersed 8-cycle. Since $8<2q$, this contradicts the conclusion of the previous step. Therefore, no immersed 4-cycle exists in this case.  
\end{proof}

The converse of Lemma~\ref{lem:T condition and the links} is not true in general. For example, let $\Theta$ be the wheel graph with four vertices: begin with a triangle whose vertices are labeled $u_1, u_2$, $u_3$, then add a central vertex $u_4$ and edges joining $u_4$ to each $u_i$ for $1\leq i\leq 3$. Let $\Gamma_1$ be the subgraph of $\Theta$ obtained by removing the edge $\Set{u_1,u_2}$, and let $\Gamma_2$ be the subgraph obtained by removing the edge $\Set{u_2,u_3}$. Set $\Gamma = \Gamma_1\coprod \Gamma_2$; see Figure~\ref{fig:A non T(4) graphical presentation}.
\begin{figure}
    \centering
    \includegraphics[width=0.7\linewidth]{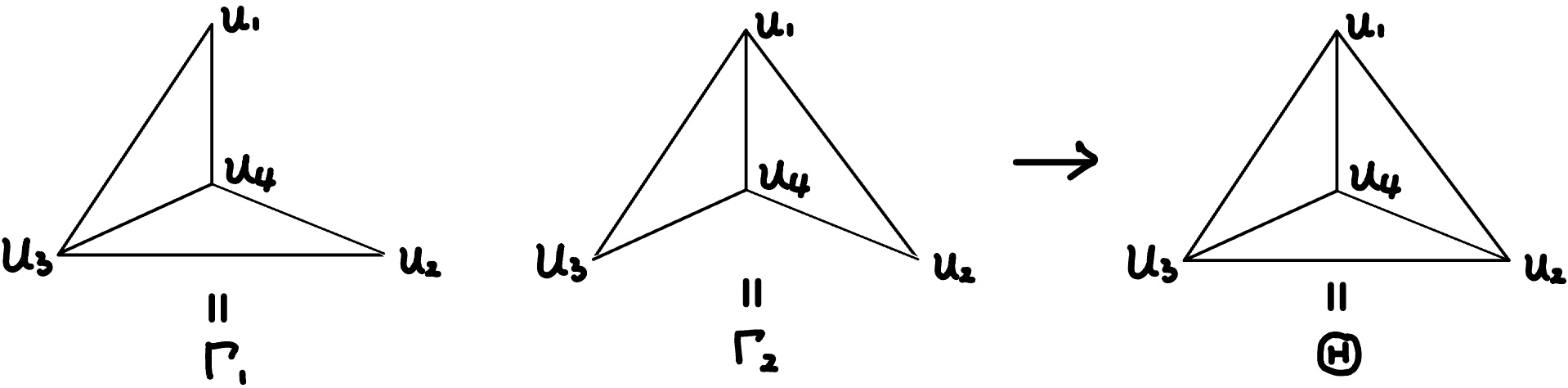}
    \caption{A graphical presentation that is not $T(4)$}
    \label{fig:A non T(4) graphical presentation}
\end{figure}
Let $f: \Gamma \rightarrow \Theta$ be the canonical map, and let $X$ be the associated non-thickened graphical complex. 
Observe that every immersed cycle in $Lk(u_i,X)$ has length a multiple of $4$; in particular, every such cycle has length either $4$ or at least $8$. 
However, by filling the triangles $u_1u_2u_4$, $u_2u_3u_4$, and $u_1u_3u_4$, one obtains a reduced disk diagram in the associated thickened graphical complex $X_t$. This diagram contains an interior vertex of degree 3, and hence $X_t$ does not satisfy the $T(4)$ condition. 

\section{Proof of the results}
\label{sec: Proof of the results}

In this section, we equip a non-thickened graphical complex with a metric of piecewise constant curvature and prove Theorem \ref{thm: simplicial + T(6) implies CAT(0)}. We also establish a slightly stronger statement: instead of assuming the girth conditions in Theorem \ref{thm: simplicial + T(6) implies CAT(0)}, we assume the corresponding $C(p)$ small cancellation condition. This result is stated as Corollary \ref{cor: C(3) + T(6) implies CAT(0)}. Finally, we explain how the remaining corollaries follow from the nonpositive curvature structures constructed here together with existing results in the literature. 

We begin by recalling how to endow a 2-complex with a piecewise Euclidean (or hyperbolic) metric, and the link condition for such complexes, which characterizes when such a metric is locally CAT(0) (respectively, CAT(-1)). For further details, see \cite{BH13}.

Let $X$ be a 2-complex with a cell decomposition $\coprod_\lambda e^n_\lambda$. 
For each $n$-cell $e^n_\lambda$, there is a canonical map from an $n$-dimensional polytope to $X$ whose image is $e^n_\lambda$. In particular, $X$ can be realized as a quotient of the disjoint union of these polytopes. 
Equip each cell $e^n_\lambda$ with a Euclidean (or hyperbolic) metric by embedding its corresponding polytope into $\mathbb{E}^2$ (respectively, $\mathbb{H}^2$) so that its image is the convex hull of the images of its vertices. 
Let $\varphi^n_\lambda$ denote the map from the polytope in $\mathbb{E}^2$ (or $\mathbb{H}^2$) to $X$. 
Assume the following compatibility condition: whenever an $m$-cell $e^m_\eta$ is a face of an $n$-cell $e^n_\lambda$, the map $(\varphi^n_\lambda)^{-1}\circ\varphi_\eta^m$ is an isometric embedding. 
Under this assumption, the cell-wise metrics induce a pseudometric $d$ on $X$. 
Moreover, if only finitely many isometric types of cells occur in $X$, then $d$ is a genuine metric, and $(X,d)$ is a complete metric space. 

With a piecewise Euclidean (or hyperbolic) metric $d$ as described above, the \textbf{geometric link} of a vertex $v$ in $X$ is the link graph $Lk(v,X)$ where each edge is assigned the metric length equal to the angle of the corresponding corner of the 2-cell (viewed in $\mathbb{E}^2$ or $\mathbb{H}^2$). 
The complex $X$ is said to satisfy the \textbf{link condition} if, for every vertex $v\in X$, every injective cycle in $Lk(v,X)$ has metric length at least $2\pi$. 
It is known that $(X,d)$ is locally CAT(0) (respectively, CAT(-1)) if and only if it satisfies the link condition; see \cite[Lemma 5.6]{BH13}.  

\begin{proof}[Proof of Theorem \ref{thm: simplicial + T(6) implies CAT(0)}]\label{proof: main theorem}
    Let $\langle f:\Gamma \rightarrow \Theta\rangle$ be a $T(q)$ graphical presentation with $Girth(\Gamma)\geq p$, and let $X$ be the associated non-thickened graphical complex. 
    Recall that $X = \coprod_\lambda R^\lambda_3 / \sim$ can be realized as a quotient of a disjoint union of triangles. 
    Following the notation introduced in the paragraph preceding Lemma \ref{lem:T condition and the links}, for each $\lambda$, each triangle $R^\lambda_3$ has exactly two edges mapped to $E_{tip}(X)$ and one edge mapped to $E(\Theta)$. 
    For every $v\in V_{tip}(X)$, the link $Lk(v,X)$ is isomorphic to the component of $\Gamma$ corresponding to $v$.
    Hence, $Girth(\Gamma)\geq p$ implies $Girth\big(Lk(v,X)\big)\geq p$. 
    For $v\in V(\Theta)$, Lemma \ref{lem:T condition and the links} shows that if $\langle f:\Gamma \rightarrow \Theta\rangle$ is a $T(q)$ graphical presentation with $q\geq 5$, then $Girth\big(Lk(v,X)\big)\geq 2q$.
    
    For (1), endow each 2-cell $R_3^\lambda$ with the metric of a Euclidean triangle where the two edges mapped to $E_{tip}(X)$ have length 1, and the edge mapped to $E(\Theta)$ has length $\sqrt{3}$. 
    With this choice, for every $v\in V_{tip}(X)$, each edge in the geometric link $Lk(v,X)$ has metric length $\frac{2\pi}{3}$, while for all $v\in V(\Theta)$, every edge in $Lk(v,X)$ has metric length $\frac{\pi}{6}$.
    Thus, with the induced metric $d$, the complex $(X,d)$ satisfies the link condition and is locally CAT(0). 

    For (2), it suffices to metrize every 2-cell in $X$ as a fixed isosceles geodesic triangle in $\mathbb{H}^2$ whose vertex angle is $\frac{2\pi}{p}$ and whose two base angles are $\frac{\pi}{q}$. Such a triangle exists since $\frac{1}{p}+\frac{1}{q}<\frac{1}{2}$. 
\end{proof}

For a $T(q)$ graphical presentation $\langle f: \Gamma \rightarrow \Theta\rangle$ with $q\geq 5$, Lemma \ref{Lem: pieces have length 1} implies that all pieces in the associated thickened graphical complex $X_t$ have length 1. Thus, for these graphical presentations, the $C(p)$ condition is weaker than requiring $Girth(\Gamma)\geq p$. In the following corollary, we show that the hypothesis $Girth(\Gamma)\geq p$ in Theorem \ref{thm: simplicial + T(6) implies CAT(0)} can be relaxed to the $C(p)$ condition.

\begin{corollary}\label{cor: C(3) + T(6) implies CAT(0)}
    Let $\langle f:\Gamma\rightarrow \Theta\rangle$ be a $C(p)$-$T(q)$ graphical small cancellation presentation, then the associated non-thickened graphical complex $X$ can be endowed with a locally CAT(0) metric when $(p,q)=(3,6)$ and a locally $CAT(-1)$ metric when $(p,q)\in \{(3,7), (4,5)\}$.
\end{corollary}

\begin{proof}\label{proof: C(p)-T(q) implies CAT(0)}

    To begin, we describe subdivisions of $\Gamma$ and $\Theta$, such that the subdivision of $\Gamma$ has girth at least $p$. 
    Define 
    \begin{align*}
        E_1(\Theta) &\coloneqq \{e\in E(\Theta): f^{-1}(e) \ \text{is a single edge}\} \\
        E_1(\Gamma) &\coloneqq \{e\in E(\Gamma): f(e)\in E_1(\Theta)\}.
    \end{align*}
    Subdivide every edge in $E_1(\Theta)$ and in $E_1(\Gamma)$ into $p$ edges, and denote the resulting subdivisions of $\Gamma$ and $\Theta$ by $\Gamma'$ and $\Theta'$, respectively. 
    There is a canonical graph immersion $f': \Gamma'\rightarrow \Theta'$ induced by $f$, and the non-thickened graphical complex $X'$ associated to $\langle f':\Gamma'\rightarrow \Theta'\rangle$ is a subdivision of $X$. 
    Equip $X'$ with a piecewise Euclidean (or hyperbolic) metric $d$ as in the proof of Theorem \ref{thm: simplicial + T(6) implies CAT(0)}, so that every corner at a cone tip $v\in V_{tip}(X')$ has angle $\frac{2\pi}{p}$, and every other corner has angle $\frac{\pi}{q}$. 
    To show that $d$ is locally CAT(0) (respectively, CAT(-1)), it suffices to verify that $Girth(Lk(v,X'))\geq p$ for all $v\in V_{tip}(X')$, and $Girth(Lk(v,X'))\geq 2q$ for all $v\in V(\Theta')$.

    We first show that $Girth(\Gamma')\geq p$, which implies $Girth(Lk(v,X'))\geq p$ for every $v\in V_{tip}(X')$. 
    It suffices to show that every injective $k$-cycle in $\Gamma$ with $1\leq k<p$ contains at least one edge in $E_1(\Gamma)$. 
    Observe that for $e\in E(\Gamma)$, either $e\in E_1(\Gamma)$, or the image $f(e)$ can be parametrized as a piece of length 1. 
    Since $\langle f: \Gamma \rightarrow \Theta\rangle$ is a $C(p)$ graphical presentation, every immersed cycle $\tau: C\rightarrow \Gamma$ with $|C|< p$ must contain an edge in $E_1(\Gamma)$. Thus $Girth(\Gamma')\geq p$.

    Next, we show $Girth(Lk(v,X'))\geq 2q$ for every $v\in V(\Theta')$. Notice that each vertex $v\in V(\Theta')$ is either an original vertex in $\Theta$ or a subdivision vertex lying on an edge in $E_1(\Theta)$. 
    If $v$ is a subdivision vertex, then it lies on an edge in $E_1(\Theta)$, and $Lk(v,X')$ is a path of length 2. Hence $Girth(Lk(v,X'))=\infty$. 
    If $v\in V(\Theta)$, then $Lk(v,X')$ is isomorphic to $Lk(v,X)$, and by Lemma \ref{lem:T condition and the links}, this link has girth at least $2q$.

    Therefore, $X'$ satisfies the link condition, and the metric $d$ on $X$ is locally CAT(0) (respectively, CAT(-1)). Since $X'$ is a subdivision of $X$, the same metric $d$ descends to $X$ as well.
    
\end{proof}

Let $\widetilde{X}$ be a simply connected $C(p)$-$T(q)$ non-thickened graphical complex with $(p,q)\in \{(3,6),(4,5)\}$. As in the proof of Corollary \ref{cor: C(3) + T(6) implies CAT(0)}, we may subdivide $\widetilde{X}$ into a triangle complex $\widetilde{X}'$. 
When $(p,q)=(3,6)$, each 2-cell in $\widetilde{X}'$ can be metrized as a Euclidean triangle with angles $\frac{2\pi}{3}$, $\frac{\pi}{6}$, $\frac{\pi}{6}$.
When $(p,q)=(4,5)$, instead of giving every 2-cell a hyperbolic metric, we may again assign to each 2-cell a Euclidean triangle, this time with angles $\frac{\pi}{2}$, $\frac{\pi}{4}$, $\frac{\pi}{4}$. 
In both cases, the induced metric on $\widetilde{X}'$ is CAT(0), recurrent (see \cite[Section 2]{OP21-recurrent}, from Definition 2.1 to Example 2.5), and has rational angles (see the paragraph preceding Theorem 1.1 in \cite{NOP22-torsion}). 
Therefore, Corollary \ref{cor:Tits alternative} follows from either \cite[Main Theorem]{OP21-recurrent} or \cite[Theorem A]{OP22-Tits}, and Corollary \ref{cor:fixed point} follows from \cite[Theorem 1.1]{NOP22-torsion}.

\section*{Acknowledgments}
I would like to thank Damian Osajda for suggesting this project, many insightful discussions, and comments on earlier drafts. I am also grateful to Martín Blufstein for valuable discussions and assistance with proofreading, and to Tim Berland for helpful conversations. The work presented here was partially supported by the Carlsberg Foundation, grant CF23-1226.

\bibliographystyle{alpha}
\bibliography{ref}

\end{document}

%% file: preamble.tex
\usepackage{amssymb}

\usepackage{mathtools}

\usepackage{amsthm}

\usepackage{geometry}

\usepackage[dvipsnames]{xcolor}

\usepackage{comment}

\usepackage[shortlabels]{enumitem}

\usepackage{float}

\usepackage{subcaption}

\usepackage{tikz-cd}

\theoremstyle{plain}
\newtheorem{theorem}{Theorem}[section]

\newtheorem{lemma}[theorem]{Lemma}
\newtheorem{corollary}[theorem]{Corollary}

\newtheorem{example}[theorem]{Example}

\theoremstyle{definition}
\newtheorem{definition}{Definition}

\theoremstyle{remark}
\newtheorem{remark}{Remark}

\usepackage[colorlinks=true, linkcolor=blue, citecolor=blue, urlcolor=blue]{hyperref}

\DeclarePairedDelimiter\braces\{\}

\newcommand{\Set}{\braces*}